\documentclass[9pt]{amsart} %
\usepackage[T1]{fontenc}
\usepackage[utf8]{inputenc}
\usepackage[english]{babel}
\usepackage[p,osf]{cochineal}
\usepackage[scale=.95,type1]{cabin}
\usepackage[zerostyle=c,scaled=.94]{newtxtt}
\usepackage[T1]{fontenc} 
\usepackage[utf8]{inputenc} 
\usepackage{amsthm}

\usepackage{amsmath}
\usepackage{amssymb}
\usepackage{pgfplots} 
\pgfplotsset{compat=newest} 
\usepgfplotslibrary{fillbetween,colorbrewer}
\usepackage{amsfonts}
\usepackage{amsaddr}
\usepackage{xspace,subcaption}
\usepackage{enumitem} 
\usepackage{csquotes}
\usepackage{xcolor}
\usepackage[colorlinks]{hyperref}
\hypersetup{colorlinks, breaklinks, citecolor=teal,filecolor=black}
\usepackage{cleveref}
\usepackage{graphicx}
\usepackage{mathtools}
\mathtoolsset{centercolon}
\usepackage{bm}
\usepackage{pgfkeys}
\usepackage{pgf,interval}
\intervalconfig{soft open fences}
\usepackage[giveninits=true,url=false,doi=false,isbn=false,eprint=true,datamodel=mrnumber,sorting=nty,maxcitenames=4,maxbibnames=99,backref=false,block=space,backend=biber,bibstyle=phys]{biblatex} 
\AtEveryBibitem{\clearfield{month}}
\AtEveryCitekey{\clearfield{month}} 
\renewbibmacro{in:}{}
\DeclareFieldFormat[article]{title}{\emph{#1}} 
\DeclareFieldFormat{mrnumber}{\ifhyperref{\href{http://www.ams.org/mathscinet-getitem?mr=#1}{\nolinkurl{MR#1}}}{\nolinkurl{#1}}}
\DeclareFieldFormat{pmid}{\ifhyperref{\href{https://www.ncbi.nlm.nih.gov/pubmed/#1}{\nolinkurl{PMID#1}}}{\nolinkurl{#1}}}
\DeclareFieldFormat{eprint}{\ifhyperref{\href{https://arxiv.org/abs/#1}{\nolinkurl{arXiv:#1}}}{\nolinkurl{#1}}}
\renewbibmacro*{doi+eprint+url}{%
  \iftoggle{bbx:doi}{\printfield{doi}}{}
  \newunit\newblock%
  \printfield{mrnumber}%
  \newunit\newblock%
  \printfield{pmid}%
  \newunit\newblock%
  \printfield{eprint}%
  \iftoggle{bbx:url}{\usebibmacro{url+urldate}}{}}
\bibliography{ref2}
\usepackage{tikz}

\numberwithin{equation}{section} 
\newtheorem{theorem}{Theorem}[section]

\newtheorem{lemma}[theorem]{Lemma}
\newtheorem{definition}[theorem]{Definition}
\newtheorem{proposition}[theorem]{Proposition}

\usetikzlibrary{plotmarks}
\usepackage{grffile}
\usetikzlibrary{plotmarks}
\usetikzlibrary{calc,decorations.markings}

\newcommand{\1}{\mathbf{1}}

\newcommand{\R}{\mathbb{R}}  %
\ifdefined\C\renewcommand{\C}{\mathbb{C}}\else\newcommand{\C}{\mathbb{C}}\fi %
\newcommand{\N}{\mathbb{N}}  %
\newcommand{\deq}{%
  \mathrel{\vbox{\offinterlineskip\ialign{%
    \hfil##\hfil\cr
    $\scriptscriptstyle d$\cr
    \noalign{\kern.1ex}
    $=$\cr
}}}}
\newcommand{\Z}{\mathbb{Z}}  %
\newcommand{\Q}{\mathbb{Q}}  %
\newcommand{\K}{\mathbb{K}} %
\newcommand{\BV}{BV}
\usepackage{subcaption}

\renewcommand{\L}{\mathcal{L}} %

\newcommand{\landauO}[1]{\mathcal{O}\left(#1\right)} 
\newcommand{\stO}[1]{\mathcal{O}_\prec\left(#1\right)} 

\newcommand*{\defeq}{\mathrel{\vcenter{\baselineskip0.5ex \lineskiplimit0pt\hbox{\scriptsize.}\hbox{\scriptsize.}}}=}
                     
\providecommand{\norm}[1]{\left\lVert#1\right\lVert} %
\providecommand{\abs}[1]{\left\lvert#1\right\rvert} %
\providecommand{\braket}[1]{\left\langle#1\right\rangle} %

\DeclareMathOperator{\E}{\mathbf{E}}

\renewcommand{\P}{\mathbf{P}}
\DeclareMathOperator{\Tr}{Tr}

\newcommand\restr[2]{{%
  \left.\kern-\nulldelimiterspace %
  #1 %
  \vphantom{\big|} %
  \right|_{#2} %
  }}
\makeatletter
\providecommand*{\diff}%
        {\@ifnextchar^{\DIfF}{\DIfF^{}}}
\def\DIfF^#1{%
        \mathop{\mathrm{\mathstrut d}}%
                \nolimits^{#1}\gobblespace
}
\def\gobblespace{%
        \futurelet\diffarg\opspace}
\def\opspace{\let\DiffSpace\! \ifx\diffarg(\let\DiffSpace\relax\else\ifx\diffarg\let\DiffSpace\relax\else\ifx\diffarg\{\let\DiffSpace\relax\fi\fi\fi\DiffSpace}  

\title{Fluctuations of Functions of Wigner Matrices} 
\author{L\'aszl\'o Erd\H{o}s \and Dominik Schr\"{o}der} 
\address{IST Austria, Am Campus 1, 3400 Klosterneuburg, Austria}
\email{lerdos@ist.ac.at}
\email{dschroed@ist.ac.at}
\date{\today}
\subjclass[2010]{60B20, 15B52} 
\keywords{Linear eigenvalue statistics; Central limit theorem; Non-Gaussian
fluctuation} %

\begin{document}
\maketitle
\section{Introduction}
The density of states of an $N\times N$ Wigner random matrix
$H=H^{(N)}$ converges to the
Wigner semicircular law~\cite{MR77805}.
More precisely, for any continuous function $f\colon\R\to\C$
\begin{align}\lim_{N\to\infty}\frac{1}{N}\Tr f(H)=\lim_{N\to\infty}\frac
{1}{N}\sum_{k=1}^N f(\lambda_k)
= \int f(x)\,\mu_{sc}(\diff x) \label{linStats}
\end{align}
where $\lambda_1,\dots,\lambda_N$ are the (real) eigenvalues of $H$ and
$\,\mu_{sc}(\diff x)\defeq\frac{1}{2\pi}\sqrt{(4-x^2)_+}\diff x$.

It is well known that for regular functions $f$, the normalized linear
eigenvalue statistics
$\frac{1}{N} \Tr f(H)$ have an asymptotically Gaussian fluctuation on
scale of order $1/N$, see,
for example, \cite{MR3116567,MR2561434,MR2222385,MR3063494,MR2829615,MR3568789,MR3018879} 
for different results in this direction, also for other random matrix ensembles.
To our knowledge, this result under
the weakest regularity condition on $f$
was proved in \cite{MR3116567}; for general Wigner matrices
$f\in H^{1+\epsilon}$ was required, while for Wigner matrices
with substantial GUE component $f\in H^{1/2+\epsilon}$ was sufficient.
Notice that the order of the fluctuation $1/N$ is much smaller than
$1/\sqrt{N}$ which would be predicted by
the standard central limit theorem (CLT) if the eigenvalues
were weakly dependent. The failure of CLT on scale $1/\sqrt{N}$
is a signature of the strong correlations among the eigenvalues.

In this paper we investigate the individual matrix elements of $f(H)$.
We will show that the semicircle law \eqref{linStats}
holds also for any diagonal matrix element $f(H)_{ii}$ and not only for
their average, $\frac{1}{N} \Tr f(H)$; however, the
corresponding fluctuation is much larger, it is on scale $1/\sqrt{N}$.
Moreover, the limiting distribution of the rescaled fluctuation
is not necessarily Gaussian; it also depends on the distribution of the
matrix element $h_{ii}$. Similar fluctuation results hold for
the off diagonal matrix elements $f(H)_{ij}$, $i\ne j$. For regularity
condition, we merely assume that $f$ is of bounded variation,
$f\in BV$.
We also prove an effective error bound of order $N^{-2/3}$ that we can
improve to
$N^{-1}$
if $f'\in L^\infty$,
i.e. we provide
a two-term expansion for each matrix element of $f(H)$.

Similar results (with less precise error bounds) were obtained
previously in \cite{MR2489497} for Gaussian random matrices and in \cite
{1103.2345,MR3090549,MR2880032} for general Wigner matrices
under the much stronger regularity assumptions that
\begin{align}
\quad \int_\R(1+\abs{\xi})^3 \abs{\widehat f(\xi)}\diff\xi<\infty\quad
\text{or}\quad\int_\R(1+\abs{\xi})^{2s}\abs{\widehat f(\xi)}^2\diff\xi
<\infty\quad\text{for some }s>3,\label{stronger regularity}
\end{align}
where $\widehat f(\xi)\defeq\int_\R e^{-i\xi x}f(x)\diff x$.
The main novelty of the current work is thus to relax these regularity
conditions to $f\in BV$. In addition, \cite
{1103.2345,MR3090549,MR2880032} assumed that in the case of
complex Hermitian matrices, the real and imaginary part of the entries
have equal variance. Our approach does not require this technical assumption.
We also refer to \cite{MR3155024} where similar questions have
been studied for more general statistics of the form $\Tr[f(H)A]$ for
non-random matrices $A$ under the fairly strong regularity condition
$\int(1+\abs{\xi})^4\lvert\widehat f(\xi)\rvert\diff\xi<\infty$.

A special case of these questions is when the test function $f(x)$ is
given by $\varphi_z(x) = (x-z)^{-1}$ for some complex parameter $z$
in the upper half plane, $\eta\defeq\Im z>0$. In fact, for $f$ which
are analytic in a complex neighborhood of $[-2,2]$, a simple contour
integration shows that for the linear statistics
it is sufficient to understand the resolvent of $H$, i.e., $\varphi
_z(H)=(H-z)^{-1}$ for
any fixed $z$ in the upper half plane. If $f$ is less regular, one may
still express $f(H)$
as an integral of the resolvents over $z$, weighted by the $\partial
_{\bar z}$-derivative of an almost analytic extension of $f$
to the upper half plane (Helffer-Sj\"ostrand formula). In this case,
the integration effectively involves
the regime of $z$ close to the real axis, so the resolvent $(H-z)^{-1}$
and its matrix elements need to be
controlled even as $\eta\to0$ simultaneously with $N\to\infty$. These
results are commonly called {\it local semicircle laws}.
They hold down to the optimal scale $\eta\gg1/N$ with an optimal error
bound of order $1/\sqrt{N\eta}$ for
the individual matrix elements and a bound of order $1/N\eta$ for the
normalized trace of the resolvent
(see, e.g. \cite{MR2871147}). With the help of the Helffer-Sj\"ostrand
formula, more accurate local laws can be transformed to weaker
regularity assumptions on the test function in the linear eigenvalue
statistics, see \cite{MR3116567}.
In this paper we replace the Helffer-Sj\"ostrand formula by Pleijel's
formula \cite{MR167751} that provides
a more effective functional calculus for functions with low regularity.

A similar relation between regularity and local laws holds for
individual matrix elements, $f(H)_{ii}$. Using the
Schur complement formula one can relate $f(H)_{ii}$ to the \emph
{difference} of a linear statistics for $H$
and for its minor $\widehat H$ obtained by removing the $i$-th row and
column from $H$.
In a recent paper \cite{MR3805203} we investigated the
fluctuations of this difference without
directly connecting it to $f(H)_{ii}$.
Applied to a special family of test function $f(x) =\abs{x-a}$, the
difference of linear statistics is closely related to
the fluctuation of Kerov's interlacing sequences of the eigenvalues of
$H$ and its minor.

Motivated by this application, Sasha Sodin pointed out that this
fluctuation can be related to the fluctuation of a single matrix
element of the resolvent by the Markov correspondence, see \cite{MR3733197}
for details. It is therefore natural to ask
if one could use the fluctuation result from \cite{MR3805203}
on the interlacing sequences to strengthen the existing results
on the fluctuations of the matrix elements of the resolvent and hence
of $f(H)$. In fact, not the result itself, but the core of
the analysis in \cite{MR3805203} can be applied; this is the
content of the current paper.
We thank Sasha for asking this question and calling our attention to
the problem of fluctuation
of the matrix elements of $f(H)$ and to the previous literature \cite
{MR2489497,1103.2345,MR3090549,MR2880032}.
Furthermore, he pointed out to us that the contour integral formula
from Pleijel's paper \cite{MR167751} could potentially
replace the
Helffer-Sj\"ostrand formula in our argument to the end of further
reducing the regularity assumptions on $f$.
We are very grateful to him for this insightful idea that we believe
will have further applications.

\section{Main results}
We consider complex Hermitian and real symmetric random $N\times N$
matrices $H=(h_{ij})_{i,j=1}^N$ with the entries being independent (up
to the symmetry constraint $h_{ij}=\overline{h_{ji}}$) random variables
satisfying
\begin{equation}\label{moments}
\E h_{ij}=0,\quad\E\abs{h_{ij}}^2=\frac{s_{ij}}{N}\quad\text
{and}\quad\E\abs{h_{ij}}^p \leq\frac{\mu_p}{N^{p/2}}
\end{equation}
for all $i,j,p$ and some absolute constants $\mu_p$. We assume that the
matrix of variances is approximately
stochastic, i.e.
\begin{align}\label{stochasticS}\sum_j s_{ij} = N+ \landauO{1}
\end{align}
to guarantee
that the limiting density of states is the Wigner semicircular law.

To formulate the error bound concisely we introduce the following
commonly used (see, e.g., \cite{MR3068390}) notion of high
probability bound.
\begin{definition}[Stochastic Domination]\label{def:stochDom}
If
\[
X=\left( X^{(N)}(u) \,\lvert\, N\in\N, u\in U^{(N)} \right)\quad\text
{and}\quad Y=\left( Y^{(N)}(u) \,\lvert\, N\in\N, u\in U^{(N)} \right)
\]
are families of random variables indexed by $N$, and possibly some
parameter $u$, then we say that $X$ is stochastically dominated by $Y$,
if for all $\epsilon, D>0$ we have
\[
\sup_{u\in U^{(N)}} \P\left[X^{(N)}(u)>N^\epsilon Y^{(N)}(u)\right
]\leq N^{-D}
\]
for large enough $N\geq N_0(\epsilon,D)$. In this case we use the
notation $X\prec Y$. Moreover, if we have $\abs{X}\prec Y$, we also
write $X=\stO{Y}$.
\end{definition}
It can be checked (see \cite[Lemma 4.4]{MR3068390}) that
$\prec$ satisfies the usual arithmetic properties, e.g.
if $X_1\prec Y_1$ and $X_2\prec Y_2$, then also $X_1+X_2\prec Y_1 +Y_2$
and $X_1X_2\prec Y_1 Y_2$.
We will say that a (sequence of) events $A=A^{(N)}$ holds with \emph
{overwhelming probability} if $\P(A^{(N)}) \ge1- N^{-D}$ for any
$D>0$ and $N\ge N_0(D)$. In particular, under the conditions \eqref
{moments}, we have $h_{ij}\prec N^{-1/2}$ and $\max_k\abs{\lambda_k}
\le3$ with overwhelming
probability.

We further introduce a notion quantifying the rate of weak convergence
of distributions. We say that a sequence of random variables $X_N$ \emph
{converges in distribution at a rate $r(N)$} to $X$ if for any $t\in\R
$ it holds that
\[
\E e^{i t X_N }=\E^{itX}+\,\mathcal{O}_t\left(r(N)\right),
\]
where we allow the coefficient of the rate to be $t$-dependent
uniformly for $\abs{t}\le T$ for any fixed $T$. If $X_N$ converges in
distribution at a rate $r(N)$, we write
\[
X_N \deq X + \landauO{r(N)}.
\]
In particular, this implies that
\[
\E\Phi(X_N) = \E\Phi(X) + \landauO{r(N)}
\]
for any analytic function $\Phi$ with compactly supported Fourier transform.

Our main result for the diagonal entries of $f(H)$ is summarized in the
following theorem. By permutational symmetry there is no loss in
generality in studying $f(H)_{11}$. By considering real and imaginary
parts separately, from now on we always assume that $f$ is real valued.

\begin{theorem}\label{mainThmDiag}
Let the Wigner matrix $H$ satisfy \eqref{moments}, $s_{ij}=1$ for
$i\not= j$ and $s_{ii}\le C$ for all $i$, $\E\abs{h_{1j}}^4=\sigma
_4/N^2$ for $j=2,\dots,N$ and $\E h_{ij}^2=\sigma_2/N$ with some $\sigma
_2, \sigma_4 \in\R$. Moreover, let $f\in\BV([-3,3])$
be some real-valued function of bounded variation and assume that
$h_{11}\deq\xi_{11}/\sqrt N$ where $\xi_{11}$ is an $N$-independent
random variable. Then
\begin{align}\label{eq:mainTHM}f(H)_{11} \deq\int f(x)\,\mu_{sc}(\diff
x) + \frac{\widehat\Delta_{f} + \xi_{11} \int f(x)x\,\mu_{sc} (\diff
x)}{\sqrt N} +
\begin{cases}\landauO{N^{-1}}&\text{if }f'\in L^\infty,\\ \landauO
{N^{-2/3}}&\text{else},
\end{cases}
\end{align}
where $\widehat\Delta_{f}$ is a centered Gaussian random variable of
variance
\begin{align}\label{variance}
\E\left(\widehat\Delta_f\right)^2 =V_{f,1}+V_{f,1}^{(\sigma_2)}-
2V_{f,2} - (1+\sigma_2)V_{f,3} + (\sigma_4-2-\sigma_2^2)V_{f,4},
\end{align}
and the $V_{f,i}$ and $V_{f,1}^{(\sigma_2)}$ are given by quadratic
forms defined in \eqref{Vi}.

More precisely, \eqref{eq:mainTHM} means that, to leading order
\begin{align}f(H)_{11}=\int f(x)\,\mu_{sc}(\diff x)+\stO{N^{-1/2}}\label{lolnd}
\end{align}
and, weakly
\begin{align}\label{weak}
T_f^{(N)}\defeq\sqrt N\left[f(H)_{11} - \int f(x)\,\mu_{sc}(\diff
x)\right]-\xi_{11}\int f(x)x\,\mu_{sc}(\diff x) \Rightarrow\widehat
\Delta_f
\end{align}
at a speed
\[
\E\left(T_f^{(N)}\right)^k = \E\widehat\Delta_f^k+
\begin{cases}\landauO{\frac{ C^k (k/2)!}{\sqrt{N}}}&\text{if }f'\in
L^\infty,\\
\landauO{\frac{C^k(k/2)!}{N^{1/6}}} &\text{else}
\end{cases}
\]
for all $k$. The speed of convergence in the L\'evy metric $d_L$ is
given by
\begin{align}\label{speed}
d_{L} (T_f^{(N)}, \widehat\Delta_f )\le C(f) \frac{\log\log N}{\sqrt
{\log N}}
\end{align}
with some constant depending on $f$.
\end{theorem}

The corresponding result for the off diagonal terms is as follows.
\begin{theorem}\label{mainThmOffDiag}
Under the assumptions of Theorem \ref{mainThmDiag},
\begin{align}\label{eq:mainTHMod}f(H)_{12} \deq\frac{1}{\sqrt N} \left
[\widetilde\Delta_{f} + \xi_{12} \int f(x)x\,\mu_{sc} (\diff x)\right]+
\begin{cases}\landauO{N^{-1}}&\text{if }f'\in L^\infty,\\ \landauO
{N^{-2/3}}&\text{else},
\end{cases}
\end{align}
where $\widetilde\Delta_{f}$ is a centered complex Gaussian satisfying
\[
\E\widetilde\Delta_{f}^2 = V_{f,1}^{(\sigma_2)} - V_{f,2} -\sigma
_2V_{f,3},\quad\E\abs{\widetilde\Delta_{f}}^2 = V_{f,1} - V_{f,2}-V_{f,3}.
\]
and the $V_{f,i}$ and $V_{f,1}^{(\sigma_2)}$ are defined in \eqref{Vi}.

More precisely, \eqref{eq:mainTHMod} means that
\begin{align}f(H)_{12}=\stO{N^{-1/2}}\label{lolnod}
\end{align}
and, introducing the notation
\[
S_f^{(N)}\defeq\sqrt N f(H)_{12}-\xi_{12}\int f(x)x\,\mu_{sc}(\diff x),
\]
we have that
\begin{align*} \E\left( S_f^{(N)} \right)^k\left(\overline{S_f^{(N)}
}\right)^l = \E\widetilde\Delta_f^k \overline{\widetilde\Delta_f}^l +
\begin{cases}\landauO{\frac{((k+l)/2)!}{\sqrt{N}}}&\text{if }f'\in
L^\infty,\\
\landauO{\frac{((k+l)/2)!}{N^{1/6}}} &\text{else}
\end{cases}
\end{align*}
holds for all $k,l\in\N$. The analogues of \eqref{weak}
and \eqref{speed} also hold for $T_f^{(N)}$ replaced with $S_f^{(N)}$.
\end{theorem}
The fluctuation results in Theorems \ref{mainThmDiag} and \ref
{mainThmOffDiag} for test functions satisfying the stronger regularity
assumption \eqref{stronger regularity} and without explicit error terms
have been proven in \cite{1103.2345,MR3090549}.

We also remark that \eqref{weak} implies the joint asymptotic normality
of the fluctuations of $f(H^{(N)})_{11}$ for
several test functions.
More precisely, for any $f\in\BV$ we define $T_f^{(N)}$ via \eqref{weak}.
Then for any given functions $f_1, f_2, \dots, f_k\in\BV$, the random
$k$-vector
\[
\left( T_{f_1}^{(N)}, T_{f_2}^{(N)}, \dots, T_{f_k}^{(N)} \right)
\]
weakly converges to a Gaussian vector with covariance given via the
variance \eqref{variance} using the parallelogram identity.
Similar result holds for the joint distribution of the off diagonal
elements $f_k(H)_{12}$. One may specialize this result to the case when
$f$ is a characteristic function, i.e. we may define
\[
T_x^{(N)} \defeq T_{\1_{[-3,x]}}^{(N)}, \qquad x\in[-3,3],
\]
where $\1_{[a,b]}$ is the characteristic function of the interval
$[a,b]$. Clearly, the finite dimensional marginals
of the sequence of stochastic processes
$\{ T_x^{(N)}, x\in[-3,3]\}$
are asymptotically Gaussian. The tightness remains an open question.

\section{Pleijel's inversion formula}
Our main tool relating $f(H)_{ij}$ to the resolvent $G=G(z)=(H-z)^{-1}$
is summarized in the following proposition. We formulate it for general
probability measures $\mu$ supported on some $[-K,K]$ and their
Stieltjes transform
\[
m_\mu(z)=\int\frac{1}{\lambda-z}\, \mu(\diff\lambda).
\]
Later we will apply the proposition to $\mu=\rho_N$ and $\mu=\widetilde
\rho_N$ with $\rho_N$, $\widetilde\rho_N$ being the spectral measures
of typical diagonal and off-diagonal entries
\[
\int f\diff\rho_{N}=f(H)_{11}, \quad\int f\diff\widetilde\rho
_{N}=f(H)_{12}.
\]

\begin{proposition}\label{prop_int_pleijel}
Let $L>K>0$ and let $\mu$ denote a probability measure which is
supported on $[-K,K]$ and let $f\in\BV([-L,L])$ be a function of
bounded variation which is compactly supported in $[-L,L]$.
Then
\begin{align}\label{eq:int_pleijel}\int f(\lambda)\,\mu(\diff\lambda
)&=\frac{1}{2\pi}\iint_{I_{\eta_0}^M} m_\mu(x+i\eta)\diff\eta\diff
f(x)+\frac{1}{\pi}\int_{-L}^L f(x)\Im m_\mu(x+M i)\diff x\\\nonumber
&\qquad\qquad+\landauO{\eta_0\norm{m_\mu(\cdot+i\eta_0)}_{L^1(\abs
{\diff f})}}\\ &=\frac{1}{2\pi}\iint_{I_{\eta_0}^M} m_\mu(x+i\eta)\diff
\eta\diff f(x)+\landauO{\eta_0\norm{m_\mu(\cdot+i\eta_0)}_{L^1(\abs
{\diff f})}+\frac{1}{M}\norm{f}_{1}}\nonumber
\end{align}
holds for any $\eta_0,M>0$ where $I_{\eta_0}^M\defeq[-L,L]\times
([-M,M]\setminus[-\eta_0,\eta_0])$, $\norm{\cdot}_{1}=\norm{\cdot
}_{L^{1}(\diff x)}$ and $\diff f$ is understood as the (signed)
Lebesgue--Stieltjes measure.
\end{proposition}
Before going into the proof, we present a special case of Proposition
\ref{prop_int_pleijel}. If $f=\1_{[x,x']}$, then \eqref{eq:int_pleijel}
can be written as the path integral
\begin{align}
\mu([x,x']) &= \frac{1}{2\pi i} \int_{\gamma(x,x')} m_\mu(z)\diff z
+\landauO{\eta_0 [\abs{m_\mu(x+i\eta_0)}+\abs{m_\mu(x'+i\eta_0)}]}
, \label{pleijel_int}
\end{align}
where $\gamma(x,x')$ is the chain indicated in Figure \eqref
{gammaCurve}. We also want to remark that for our purposes \eqref
{eq:int_pleijel} is favorable over the Helffer-Sj\"{o}strand
representation, as used in \cite{MR3805203}, since it
requires considerably less regularity on $f$.

\begin{proof}[Proof of Proposition \ref{prop_int_pleijel}]
From \cite[Eq.~(5)]{MR167751} we know that
\begin{align}\label{pleijel}\mu([-K,x))=\frac{1}{2\pi i}\int_{L(x)}
m_\mu(z)\diff z+\frac{\eta_0}{\pi}\Re m_\mu(z_0) + \landauO{\eta_0 \Im
m_\mu(z_0)},
\end{align}
where $L(x)$ is a directed path as indicated in Figure \ref{curve} and
$z_0=x+i\eta_0$, $\eta_0>0$.

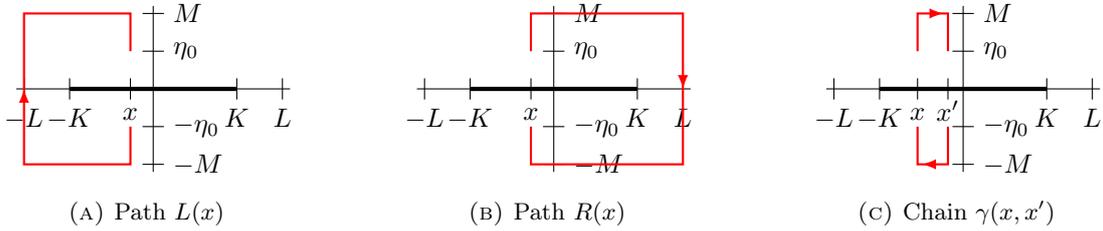
\begin{figure}[hbt]
\centering
\begin{subfigure}[b]{.31\linewidth}
\centering
\begin{tikzpicture}
\draw(0,-1.1) -- (0,1.1); \draw(-1.8,0) -- (1.8,0); %
\draw(-1.1,-4pt) -- (-1.1,4pt) node[pos=0,below] {$-K$};
\draw(1.1,-4pt) -- (1.1,4pt) node[pos=0,below] {$K$};
\draw(-0.3,-4pt) -- (-0.3,4pt) node[pos=0,below] {$x$};
\draw(-1.7,-4pt) -- (-1.7,4pt) node[pos=0,below] {$-L$};
\draw(1.7,-4pt) -- (1.7,4pt) node[pos=0,below] {$L$};
\draw(-4pt,-1.0) -- (4pt,-1.0) node[right] {$-M$};
\draw(-4pt,1.0) -- (4pt,1.0) node[right] {$M$};
\draw(-4pt,-0.5) -- (4pt,-0.5) node[right] {$-\eta_0$};
\draw(-4pt,0.5) -- (4pt,0.5) node[right] {$\eta_0$};
\draw[thick,red,xshift=0pt,decoration={ markings, mark=at position 0.5
with {\arrow{latex}}}, postaction={decorate}]
(-0.3,-0.5) -- (-0.3,-1.0) -- (-1.7,-1.0) -- (-1.7,1.0) -- (-0.3,1.0)
-- (-0.3,0.5);
\draw[ultra thick,black,xshift=0pt, postaction={decorate}]
(-1.1,0) -- (1.1,0);
\end{tikzpicture}
\caption{Path $L(x)$}\label{curve}
\end{subfigure}
\begin{subfigure}[b]{.31\linewidth}
\centering
\begin{tikzpicture}
\draw(0,-1.1) -- (0,1.1); \draw(-1.8,0) -- (1.8,0); %
\draw(-1.1,-4pt) -- (-1.1,4pt) node[pos=0,below] {$-K$};
\draw(1.1,-4pt) -- (1.1,4pt) node[pos=0,below] {$K$};
\draw(-0.3,-4pt) -- (-0.3,4pt) node[pos=0,below] {$x$};
\draw(-1.7,-4pt) -- (-1.7,4pt) node[pos=0,below] {$-L$};
\draw(1.7,-4pt) -- (1.7,4pt) node[pos=0,below] {$L$};
\draw(-4pt,-1.0) -- (4pt,-1.0) node[right] {$-M$};
\draw(-4pt,1.0) -- (4pt,1.0) node[right] {$M$};
\draw(-4pt,-0.5) -- (4pt,-0.5) node[right] {$-\eta_0$};
\draw(-4pt,0.5) -- (4pt,0.5) node[right] {$\eta_0$};
\draw[thick,red,xshift=0pt,decoration={ markings, mark=at position 0.5
with {\arrow{latex}}}, postaction={decorate}]
(-0.3,0.5) -- (-0.3,1.0) -- (1.7,1.0) -- (1.7,-1.0) -- (-0.3,-1.0) --
(-0.3,-0.5);
\draw[ultra thick,black,xshift=0pt, postaction={decorate}]
(-1.1,0) -- (1.1,0);
\end{tikzpicture}
\caption{Path $R(x)$}\label{Rcurve}
\end{subfigure}
\begin{subfigure}[b]{0.31\linewidth}
\centering
\begin{tikzpicture}
\draw(0,-1.1) -- (0,1.1); \draw(-1.8,0) -- (1.8,0); %
\draw(-1.1,-4pt) -- (-1.1,4pt) node[pos=0,below] {$-K$};
\draw(1.1,-4pt) -- (1.1,4pt) node[pos=0,below] {$K$};
\draw(-0.6,-4pt) -- (-0.6,4pt) node[pos=0,below] {$x$};
\draw(-1.7,-4pt) -- (-1.7,4pt) node[pos=0,below] {$-L$};
\draw(1.7,-4pt) -- (1.7,4pt) node[pos=0,below] {$L$};
\draw(-4pt,-1.0) -- (4pt,-1.0) node[right] {$-M$};
\draw(-4pt,1.0) -- (4pt,1.0) node[right] {$M$};
\draw(-4pt,-0.5) -- (4pt,-0.5) node[right] {$-\eta_0$};
\draw(-4pt,0.5) -- (4pt,0.5) node[right] {$\eta_0$};
\draw(-0.2,-4pt) -- (-0.2,4pt) node[pos=0.4,below] {$x'$};
\draw[thick,red,xshift=0pt,decoration={ markings, mark=at position 0.6
with {\arrow{latex}}}, postaction={decorate}]
(-0.6,0.5) -- (-0.6,1.0) -- (-0.2,1.0) -- (-0.2,0.5);
\draw[thick,red,xshift=0pt,decoration={ markings, mark=at position 0.6
with {\arrow{latex}}}, postaction={decorate}]
(-0.2,-0.5) -- (-0.2,-1.0) -- (-0.6,-1.0) -- (-0.6,-0.5);
\draw[ultra thick,black,xshift=0pt, postaction={decorate}]
(-1.1,0) -- (1.1,0);
\end{tikzpicture}
\caption{Chain $\gamma(x,x')$}\label{gammaCurve}
\end{subfigure}
\caption{Integration paths}
\end{figure}

By the definition of the Lebesgue--Stieltjes integral for functions of
bounded variation we have that
\[
\int f(\lambda)\,\mu(\diff\lambda)=\int_{-L}^L\left(\int\1(\lambda
\geq x)\,\mu(\diff\lambda)\right) \diff f(x) =\int_{-L}^L \mu([x,K])
\diff f(x).
\]
By virtue of \eqref{pleijel} we can write
\[
\int f(\lambda)\,\mu(\diff\lambda)= \frac{1}{\pi}\int_{-L}^L\left(\frac
{1}{2i}\int_{R(x)} m_\mu(z)\diff z \right)\diff f(x)+\landauO{\eta
_0\norm{m_\mu(\cdot+i\eta_0)}_{L^1(\abs{\diff f})}},
\]
where $R(x)$ is the path indicated in Figure \ref{Rcurve} and $\abs
{\diff f}$ indicates the total variation measure of $\diff f$.
We then write out the inner integral as
\begin{align*}
\frac{1}{2 i}\int_{R(x)}m_\mu(z)\diff z = \int_{\eta_0}^M \Re m_\mu(x+i
\eta)\diff\eta+\int_{x}^L \Im m_\mu(y+iM)\diff y -\int_{0}^M \Re m_\mu
(L+i\eta)\diff\eta.
\end{align*}
Since the last term is $x$-independent, it will vanish after
integrating against $\diff f$ since we assumed $f$ to be compactly
supported. For the second term we find
\begin{align*}\int f(\lambda)\,\mu(\diff\lambda)&=\frac{1}{\pi}\int
_{-L}^L \int_{\eta_0}^M \Re m_\mu(x+i\eta)\diff\eta\diff f(x)+\frac
{1}{\pi}\int_{-L}^L f(x)\Im m_\mu(x+iM)\diff x\\ &\qquad\qquad+\landauO
{\eta_0\norm{m_\mu(\cdot+i\eta_0)}_{L^1(\abs{\diff f})}}.\label
{integrated_pleijel}
\end{align*}
Since $\abs{\Im m_\mu(x+iM)}\leq1/M$ we thus have
\begin{align*}\int f(\lambda)\,\mu(\diff\lambda)\,{=}\,\frac{1}{\pi}\int
_{-L}^L \int_{\eta_0}^M \Re m_\mu(x\,{+}\,i\eta)\diff\eta\diff f(x)\,{+}\,\landauO
{\eta_0\norm{m_\mu(\cdot\,{+}\,i\eta_0)}_{L^1(\abs{\diff f})}{+}\,\frac{1}{M}\norm{f}_{1}}
\end{align*}
for any $\eta_0,M>0$. For applications it turns out to be favorable to
get rid of the real part which we can by noting that $2\Re m_\mu
(z)=m_\mu(z)+m_\mu(\overline z)$ and therefore
\begin{align*}\int f(\lambda)\,\mu(\diff\lambda)=\frac{1}{2\pi}\iint
_{I_{\eta_0}^M} m_\mu(x+i\eta)\diff\eta\diff f(x)+\landauO{\eta_0\norm
{m_\mu(\cdot+i\eta_0)}_{L^1(\abs{\diff f})}+\frac{1}{M}\norm{f}_{1}},
\end{align*}
where we recall $I_{\eta_0}^M=[-L,L]\times([-M,M]\setminus[-\eta
_0,\eta_0])$.
\end{proof}

We finally note that a variant of Proposition \ref{prop_int_pleijel}
could also be proven directly without appealing to the contour
integration from \cite{MR167751}. The key computation in that
direction is summarized in the following Lemma which we establish here
for later convenience.
\begin{lemma}\label{Stokes}
Let $f\in\BV([-L,L])$ be compactly supported and let $g$ be a function
which is analytic away from the real axis and satisfies $g(\overline
z)=\overline{g(z)}$. Then for any $\eta_0,M>0$ we have that
\begin{align*}\frac{1}{2\pi}\iint_{I_{\eta_0}^M} g(x+i\eta)\diff\eta
\diff f(x) = \frac{1}{\pi} \int_{-L}^L f(x)\Im g(x+i\eta_0)\diff x +
\landauO{\norm{f}_{1} \max_{x\in[-L,L]}\abs{g(x+iM)}}.
\end{align*}
\end{lemma}
Applying Lemma \ref{Stokes} to $g=m_\mu$ yields, modulo an error term,
\[
\frac{1}{2\pi}\iint_{I_{\eta_0}^M} m_\mu(x+i\eta)\diff\eta\diff f( x)
\approx\int\int_{-L}^L f(x)\frac{1}{\pi}\frac{\eta_0}{(\lambda-x)+\eta
_0^2}\diff x \,\mu(\diff\lambda)
\]
and taking the limit $\eta_0\to0$ makes the inner integral tend to
$f(\lambda)$ in $L^1$-sense. In this way we can establish a variant of
Proposition \ref{prop_int_pleijel}, albeit with a weaker error estimate.
\begin{proof}[Proof of Lemma \ref{Stokes}]
This follows from the computation
\begin{align*}
\iint_{I_{\eta_0}^M} g(x+i\eta)\diff\eta\diff f(x) &=-i\int_{\partial
I_{\eta_0}^M} f(x) g(z)\diff z \nonumber= 2\int_{-3}^3 f(x)\Im\left
[g(x+i \eta_0)-g(x+iM)\right]\diff x \nonumber\\ & = 2\int_{-3}^3
f(x)\Im g(x+i\eta_0)\diff x +
\landauO{\norm{f}_{1} \max_{x\in[-3,3]}\abs{g(x+iM)}},
\end{align*}
where the first step follows from Stokes' or Green's Theorem.
\end{proof}

\section{Diagonal entries}
We first prove Theorem \ref{mainThmDiag} about the diagonal entries of
$f(H)$. The spectral measure corresponding to the $(1,1)$-matrix
element, $\rho_N$ defined as
\[
\int f\diff\rho_N = f(H)_{11}
\]
is concentrated in $[-2.5,2.5]$ with overwhelming probability. We can
without loss of generality assume that $f$ is compactly supported in
$[-3,3]$ since smoothly cutting off $f$ outside the spectrum does not
change the result. Applying Proposition \ref{prop_int_pleijel} to $\mu
=\rho_N$ with $K=2.5$, $L=3$, we find that (using $z=x+i\eta$,
$z_0=x+i\eta_0$)
\begin{align}\label{fH11a} f(H)_{11} &= \frac{1}{2\pi}\iint_{I_{\eta
_0}^M} G(z)_{11}\diff\eta\diff f(x)+\stO{\eta_0 \int\abs
{G(z_0)_{11}}\diff f(x) +\frac{1}{M}\norm{f}_{1}}.
\end{align}

To analyse $G(z)_{11}$ we recall the Schur complement formula
\[
G(z)_{11}=\frac{1}{h_{11}-z-\braket{h,\widehat G(z) h}},\quad\text
{where }H = \left(
\begin{matrix}
h_{11}& h^*\\ h & \widehat H
\end{matrix}
\right),\qquad\widehat G(z)\defeq(\widehat H- z)^{-1}.
\]
To study the asymptotic behavior of $G(z)_{11}$ we rely on the local
semicircle law in the averaged form
(see \cite{MR2871147} or \cite[Theorem 2.3]{MR3068390})
applied to the resolvent of the minor
\begin{align}\label{eq:localSC}
\widehat m_N(z)= \frac{1}{N}\Tr\widehat G(z) = m(z)+\stO{\frac{1}{N\abs
{\eta}}},
\end{align}
and its entry-wise form 
\begin{align}\label{eq:localSCentry}
G(z)_{ij}-\delta_{ij}m(z)\prec\frac{1}{\sqrt{N\abs{\eta}}}
\end{align}
which both hold true for all $\abs{\eta}=\abs{\Im z}>\eta_0\gg N^{-1}$.
Here $m$ denotes the Stieltjes transform of the semicircular
distribution $\mu_{sc}$, $m(z)\defeq\int(\lambda-z)^{-1}\, \mu
_{sc}(\diff\lambda)$.

Since by \eqref{eq:localSCentry},
\[
\int\abs{G(x+i\eta_0)_{11}}\diff f(x)=\int\abs{ m(x+i\eta_0)}\diff
f(x)+\stO{\int\abs{\frac{1}{\sqrt{N\eta_0}}}\diff f(x)}\prec\norm
{\diff f}
\]
for $\eta_0\gg1/N$, where $\norm{\diff f}$ is the total variation norm
of the Lebesgue--Stieltjes measure $\diff f$, we can write \eqref
{fH11a} as
\[
f(H)_{11} = \frac{1}{2\pi}\iint_{I_{\eta_0}^M} G(x+i\eta)_{11}\diff\eta
\diff f(x)+\stO{\eta_0 \norm{\diff f} +M^{-1} \norm{f}_{1}}.
\]

In order to separate the leading order contribution from the
fluctuation, we set
\[
\Phi_N(z)= G(z)_{11}=\frac{1}{h_{11}-z-\braket{h,\widehat G (z)
h}},\qquad\widehat\Phi_N(z)=\frac{1}{-z-\widehat m_N(z)},
\]
where $\widehat m_N(z)=\frac{1}{N}\Tr\widehat G(z)$ and observe that
\begin{align}\label{HatPhiEst}\widehat\Phi_N(z)= \frac
{1}{-z-m(z)}+\frac{\stO{m(z)-\widehat m_N(z)}}{-z-m(z)}=m(z)+\stO{\frac
{1}{N\abs\eta}}
\end{align}
and by expanding both terms around $[-z-m(z)]^{-1}=m(z)$,
\begin{align}\label{PhiEst}\Phi_{N}(z)-\widehat{\Phi}_N(z) =m(z)^2 \left
[\braket{h,\widehat G(z)h}-\widehat m_N(z)-h_{11}\right] + \stO{\frac
{1}{N\abs{\eta}}}.
\end{align}
Thus $\widehat\Phi_N$ describes the leading order behavior, which is
very close to a deterministic quantity, and the leading fluctuation is
solely described by $\Phi_N-\widehat\Phi_N$.
We then can write
\begin{align*}
f(H)_{11}= \Lambda_{f}^{(N)}+\frac{\Delta_{f}^{(N)}}{\sqrt{N}}+\stO{\eta
_0 \norm{\diff f}+\frac{1}{M}\norm{f}_{1}},
\end{align*}
where
\[
\Lambda_{f}^{(N)}\defeq\frac{1}{2\pi}\iint_{I_{\eta_0}^M} \widehat\Phi
_N(z)\diff\eta\diff f(x)\quad\text{and}\quad\Delta_{f}^{(N)}\defeq
\frac{1}{2\pi}\iint_{I_{\eta_0}^M}\sqrt N [\Phi_N-\widehat\Phi
_N(z)]\diff\eta\diff f(x).
\]
The reason for the normalization will become apparent later since in
this way $\Delta_{f}^{(N)}$ is an object of order $1$.

For the leading order term we use \eqref{HatPhiEst} and Proposition \ref
{prop_int_pleijel} to compute
\begin{align*}\Lambda_{f}^{(N)}&= \frac{1}{2\pi}\int_{I_{\eta_0}^M}
m(z)\diff\eta\diff f(x) + \stO{\norm{\diff f}\int_{\eta_0}^M \frac
{1}{N \eta}\diff\eta}\\&=\int f(x) \,\mu_{sc}(\diff x)+\stO{\left[\frac
{\abs{\log M}+\abs{\log\eta_0}}{N}+\eta_0\right]\norm{\diff f}+\frac
{1}{M}\norm{f}_{1}}.
\end{align*}

For the fluctuation we use \eqref{PhiEst} to compute
\begin{align}\nonumber
\Delta_f^{(N)} &= \frac{1}{2\pi}\int_{I_{\eta_0}^M} m(z)^2 \sqrt{N}\left
[\braket{h,\widehat G(z)h}-\widehat m_N(z)-h_{11}\right]\diff\eta\diff
f(x) \\
&\quad\ + \stO{\frac{\abs{\log M}+\abs{\log\eta}}{\sqrt N}\norm{\diff f}}\notag
\\\label{DeltafN}
&=\widehat\Delta_f^{(N)}-\xi_{11}\frac{1}{2\pi}\int_{I_{\eta_0}^M}
m(z)^2\diff\eta\diff f(x) + \stO{\frac{\abs{\log M}+\abs{\log\eta
}}{\sqrt N}\norm{\diff f}} \\
&=\widehat\Delta_f^{(N)}+\xi_{11}\int f(x)x\,\mu_{sc}(\diff x) + \stO
{\frac{\abs{\log M}+\abs{\log\eta}}{\sqrt N}\norm{\diff f}+\eta_0+\frac
{1}{M^2}\norm{f}_{1}} ,\nonumber
\end{align}
where the last step followed from Lemma \ref{Stokes} and
\[
\xi_{11}\,{=}\,\sqrt N h_{11},\quad\widehat\Delta_f^{(N)} \defeq\frac{1}{2\pi
} \int_{I_{\eta_0}^M} m(z)^2 X(z)\diff\eta\diff f(x),\quad X(z)\,{=}\, \sqrt
N\left[\braket{h,\widehat G(z)h}{-}\,\widehat m_N(z)\right].
\]

We now concentrate on the computation of $\E\left(\widehat\Delta
_f^{(N)}\right)^2$. We state the main estimate of $\E X(z)X(z')$ as a lemma.
\begin{lemma}\label{lemma:EXX}
Under the assumptions of Theorem \ref{mainThmDiag} it holds that
\begin{align}
\E X(z)X(z') &= \frac{m(z)^2m(z')^2}{1-m(z)m(z')}+\frac{\sigma
_2^3m(z)^2m(z')^2}{1-\sigma_2m(z)m(z')}+(\sigma_4-1)m(z)m(z') + \stO
{\frac{\Psi}{\sqrt N \Phi}},\label{eq:EX2}
\end{align}
where
\begin{align*}\Psi&\defeq\frac{1}{\sqrt{\abs{\eta\eta'}}}\left(\frac
{1}{\sqrt{\abs{\eta}}}+\frac{1}{\sqrt{\abs{\eta'}}}+\frac{1}{\sqrt{N\abs
{\eta\eta'}}}\right)\\ \Phi&\defeq\1_{\abs{x},\abs{x'}\leq2}\left(\abs
{\eta}+\abs{\eta'}+\abs{x-x'}^2\right)+\left[(\abs{x}-2)_+ +(\abs
{x'}-2)_+ \right]
\end{align*}
and $z=x+i\eta$, $z'=x'+i\eta'$.
\end{lemma}
We remark that in the $\abs{x-x'}^2$ term in $\Phi$ could be replaced
by $\abs{x-x'}$ but we will not need this stronger bound here.
\begin{proof}[Proof of Lemma \ref{lemma:EXX}]
From (35) in \cite{MR3805203} we know that
\begin{align}\label{eq:EXX}\E\left[ X(z)X(z') | \widehat H\right] =
\frac{1}{N}\sum_{i\not= j}\left(\widehat{G}_{ij} \widehat
{G}'_{ji}+\sigma_2^2 \widehat{G}_{ij} \widehat{G}'_{ji} \right) + \frac
{\sigma_4-1}{N}\sum_{i}\widehat{G}_{ii}\widehat{G}_{ii}'
\end{align}
where, $\widehat{G}_{ij}\defeq\widehat G(z)_{ij}$, $\widehat
{G}_{ij}'\defeq\widehat G(z')_{ij}$. The last term we directly
estimate as
\begin{align}\label{GG}\frac{\sigma_4-1}{N}\sum_i \widehat
{G}_{ii}\widehat{G}'_{ii}=(\sigma_4-1)m(z)m(z')+\stO{\frac{1}{\sqrt
{N\abs{\eta}}}+\frac{1}{\sqrt{N\abs{\eta'}}}+\frac{1}{N\sqrt{\abs{\eta
\eta'}}}}.
\end{align}
Furthermore, in Lemma 9 of \cite{MR3805203} self-consistent
equations for the first two terms on the rhs.~of \eqref{eq:EXX} were
derived. We recall that
\begin{align*} [1-m(z)m(z')]\frac{1}{N}\sum_{i\not= j} \widehat
{G}_{ij}\widehat{G}'_{ji} &= m(z)^2 m(z')^2 + \stO{\frac{\Psi}{\sqrt
{N}}},\\ [1-\sigma_2 m(z)m(z')]\frac{1}{N}\sum_{i\not= j} \widehat
{G}_{ij}\widehat{G}'_{ij} &= \sigma_2 m(z)^2 m(z')^2 + \stO{\frac{\Psi
}{\sqrt{N}}},
\end{align*}
Using the straightforward inequality $\abs{m(z)}\leq1- c\abs{\eta}$,
which holds for some small $c>0$ and $z$ in the compact region
$[-10,10]\times[-i,i]$, we find
\[
\abs{1-m(z)m(z')}\geq c(\abs{\eta}+\abs{\eta'}).
\]
Since $\abs{m}$ decays outside the spectrum $[-2,2]$ we have that $\abs
{m(z)}\leq1 - c'(\abs{x}-2)_+$
for $\abs{z}\le10$,
and therefore
\[
\abs{1-m(z)m(z')}\geq c'(\abs{x}-2)_++c'(\abs{x'}-2)_+.
\]
Moreover, in the remaining regime where both $\abs{\eta},\abs{\eta'}\ll
1$ and $\abs{x},\abs{x'}\leq2$, it holds that
\begin{align*}
\abs{1-m(z)m(z')}&\geq1-\Re[m(z)m(z')] = 1 - (\Re m(z))(\Re m(z'))+
(\Im m(z))(\Im m(z')) \\&\geq c''\bigg(1 - \frac{x x'}{4} \pm\frac
{\sqrt{4-x^2}\sqrt{4-x'^2}}{4}\bigg)\geq c''(x-x')^2,
\end{align*}
where the $\pm$ depends on the signs of $\eta,\eta'$ and we allow for
the constant $c''$ to change in the last inequality.
This estimate follows from the explicit formula for $m(z)$.
Putting these inequalities together, we therefore find a constant
$C>0$ such that in the compact region $[-3,3]\times[-iM,iM]$ it holds
that $C\abs{1-m(z)m(z')}\geq\Phi$ ,
from which we obtain
\begin{align}\label{trGG}
\frac{1}{N}\sum_{i\not=j}\widehat{G}_{ij}\widehat{G}'_{ji}& =\frac
{m(z)^2m(z')^2}{1-m(z)m(z')}+ \stO{\frac{\Psi}{\sqrt N \Phi}},\\
\frac{1}{N}\sum_{i\not=j}\widehat{G}_{ij}\widehat{G}'_{ij}&=\frac{\sigma
_2m(z)^2m(z')^2 }{1-\sigma_2m(z)m(z')}+\stO{\frac{\Psi}{\sqrt N \Phi
}}.\nonumber
\end{align}
Now \eqref{eq:EX2} follows from combining \eqref{eq:EXX}, \eqref{GG}
and \eqref{trGG}.
\end{proof}

Using Lemma \ref{lemma:EXX} we then compute
\begin{align*}\E\left(\widehat\Delta_f^{(N)}\right)^2&=\frac{1}{(2\pi
)^2}\iiiint_{I_{\eta_0}^M} m(z)^2m(z')^2 \E X(z) X(z')\diff\boldsymbol
{\eta}\diff f(\boldsymbol{x})\\
&= \frac{1}{(2\pi)^2}\iiiint_{I_{\eta_0}^M}\Big[\frac
{m(z)^4m(z')^4}{1-m(z)m(z')}+\frac{\sigma_2^3m(z)^4m(z')^4}{1-\sigma
_2m(z)m(z')}\\
&\qquad\qquad+(\sigma_4-1)m(z)^3m(z')^3\Big]\diff\boldsymbol{\eta
}\diff f(\boldsymbol{x})+\landauO{\iiiint_{I_{\eta_0}^M}\frac{\Psi
}{\sqrt{N}\Phi}\diff\boldsymbol{\eta}\diff f(\boldsymbol{x})},
\end{align*}
where $\diff\boldsymbol{\eta}=\diff\eta\diff\eta'$ and $\diff
f(\boldsymbol{x})=\diff f(x)\diff f(x')$.
To estimate the error term we have to compute
\begin{align*}
\iint_{-2}^2 \iint_{\eta_0}^M \frac{1}{\eta+\eta'+\abs{x-x'}^2}\frac
{1}{\sqrt{\eta\eta'}}\left(\frac{1}{\sqrt{\eta}}+\frac{1}{\sqrt{\eta
'}}+\frac{1}{\sqrt{N \eta\eta'}}\right)\diff\boldsymbol{\eta}\diff
f(\boldsymbol{x})
\end{align*}
and readily check that
\[
\iiiint_{I_{\eta_0}^M}\frac{\Psi}{\sqrt{N}\Phi}\diff\boldsymbol{\eta
}\diff f(\boldsymbol{x})\prec
\begin{cases}
(\abs{\log M}+\abs{\log\eta_0})/\sqrt N &\text{if }f'\text{ is
bounded},\\
(\abs{\log M}+\abs{\log\eta_0})/\sqrt{N\eta_0} &\text{else}.
\end{cases}
\]

By using Lemma \ref{Stokes} and organizing the contributions from the
boundary terms at $\eta_0$ and $-\eta_0$,
we find that the leading order of $\E(\widehat\Delta_f^{(N)})^2$
becomes
\begin{align}\nonumber&
\frac{1}{2\pi^2} \Re\iint_{-3}^3 f(x)f(x')\Bigg(\left[\frac
{m(z_0)^4m(\overline{z_0'})^4}{1-m(z_0)m(\overline{z_0'})}+\frac{\sigma
_2^3m(z_0)^4m(\overline{z_0'})^4}{1-\sigma_2 m(z_0)m(\overline
{z_0'})}+(\sigma_4-1)m( z_0)^3m(\overline{z_0'})^3\right]\\ &-\left
[\frac{m(z_0)^4m(z_0')^4}{1-m(z_0)m(z_0')}+\frac{\sigma
_2^3m(z_0)^4m(z_0')^4}{1-\sigma_2 m(z_0)m(z_0')}+(\sigma_4-1)m( z_0
)^3m(z_0')^3\right]\Bigg)\diff\boldsymbol x +\stO{\frac{\norm
{f}_{1}}{M^3}},\label{Eint3a}
\end{align}
where $z_0=x+i\eta_0$ and $z_0'=x'+i\eta_0$. Since
\[
\frac{a^4}{1-a}=\frac{a}{1-a}-a-a^2-a^3
\]
and for any fixed $k\in\N$
\begin{align*}
&\frac{1}{2\pi^2}\Re\iint_{-3}^3 f(x)f(x') \left[m(z_0)^k m(\overline
{z_0'})^k - m(z_0)^k m(z_0')^k\right] \diff\boldsymbol x\\
&\quad = \left(\frac
{1}{\pi}\Im\int_{-2}^2 f(x) m(x)^k\diff x\right)^2+\stO{\eta_0}
\end{align*}
we can conclude that \eqref{Eint3a} becomes
\begin{align}\nonumber
&\frac{1}{2\pi^2} \Re\iint_{-3}^3 f(x)f(x')\Bigg(\frac
{m(z_0)m(\overline{z_0'})}{1-m(z_0)m(\overline{z_0'})} - \frac
{m(z_0)m(z_0')}{1-m(z_0)m(z_0')}\Bigg)\diff\boldsymbol x \\\nonumber&
\quad+ \frac{1}{2\pi^2} \Re\iint_{-3}^3 f(x)f(x') \Bigg(\frac
{m(z_0)m(\overline{z_0'})}{1-\sigma_2 m(z_0)m(\overline{z_0'})} - \frac
{m(z_0)m(z_0')}{1-\sigma_2 m(z_0)m(z_0')}\Bigg)\diff\boldsymbol x\\
\nonumber&\quad- 2\left(\frac{1}{\pi}\Im\int_\R f(x) m(x)\diff x\right
)^2 - (1+\sigma_2)\left(\frac{1}{\pi}\Im\int_\R f(x) m(x)^2\diff x\right
)^2 \\ &\quad+ (\sigma_4-2-\sigma_2^2) \left(\frac{1}{\pi}\Im\int_\R
f(x) m(x)^3\diff x\right)^2 + \landauO{\frac{\norm{f}_{L^1}}{M^3}+\eta
_0}.\label{Eint2bb}
\end{align}
The first term of \eqref{Eint2bb} was already computed on page 17 of
\cite{MR3805203}. The computation of the second term is very
similar to the first one and the remaining terms are routine
calculations. We arrive at
\begin{align*}
\E\left(\widehat\Delta_{f}^{(N)}\right)^2 &= V_{f,1}+V_{f,1}^{(\sigma
_2)}- 2V_{f,2} - (1+\sigma_2)V_{f,3} + (\sigma_4-2-\sigma_2^2)V_{f,4}
\\ &\qquad\qquad\qquad+ \landauO{\eta_0+\frac{\norm{f}_{1}}{M^3}+\frac
{\abs{\log M}+\abs{\log\eta_0}}{\sqrt{N\eta_0}}\norm{\diff f}}
\end{align*}
in the general case and
\begin{align*}
\E\left(\widehat\Delta_{f}^{(N)}\right)^2 &= V_{f,1}+V_{f,1}^{(\sigma
_2)}- 2V_{f,2} - (1+\sigma_2)V_{f,3} + (\sigma_4-2-\sigma_2^2)V_{f,4}
\\ &\qquad\qquad\qquad+ \landauO{\eta_0+\frac{\norm{f}_{1}}{M^3}+\frac
{\abs{\log M}+\abs{\log\eta_0}}{\sqrt{N}}\norm{f'}_{L^\infty}}
\end{align*}
in the case of $f$ with bounded derivative $f'\in L^\infty([-3,3])$,
where
\begin{align}\nonumber
V_{f,1}& \defeq\int f(x)^2\,\mu_{sc}(\diff x),\,\,\\
V_{f,1}^{(\sigma
_2)}&\defeq\iint\frac{f(x)f(y)(1-\sigma_2^2)}{1-xy\sigma
_2+(x^2+y^2-2)\sigma_2^2-xy\sigma_2^3+\sigma_2^4} \,\mu_{sc}(\diff x)\,
\mu_{sc}(\diff y) \notag\\
V_{f,2}&\defeq\left(\int f(x)\,\mu_{sc}(\diff
x)\right)^2,\,
\, V_{f,3}\defeq\left(\int f(x)x\,\mu_{sc}(\diff x)\right)^2,\,\,\notag\\
V_{f,4}&\defeq\left(\int f(x)(x^2-1)\,\mu_{sc}(\diff x)\right)^2. \label{Vi}
\end{align}
We note that $V_{f,1}^{(\sigma_2)}$ simplifies to
$V_{f,1}^{(1)}=V_{f,1}$ and $V_{f,1}^{(0)}=V_{f,2}$ in the two
important cases $\sigma_2=0,1$.

We now choose $M=N$ and $\eta_0$ depending on the regularity of $f$. In
the general case of $f\in\BV([-3,3])$ it turns out that $\eta
_0=N^{-2/3}$ minimizes the error of $\E\left(\widehat\Delta
_{f}^{(N)}\right)^2$,\vadjust{\eject} whereas for $f$ with bounded derivative, a choice
of $\eta_0=N^{-1+\epsilon}$ for any small $\epsilon>0$ is optimal. Thus
\begin{align}\label{deltafN2moment}
\E\left(\widehat\Delta_{f}^{(N)}\right)^2 = \E\left(\widehat\Delta
_f\right)^2 +
\begin{cases}
\stO{N^{-1/2}} & \text{if }f'\in L^\infty([-3,3]),\\
\stO{N^{-1/6}} & \text{else.}
\end{cases}
\end{align}
where $\widehat\Delta_f$ is a centered Gaussian of variance
\begin{align*}
\E\left(\widehat\Delta_f\right)^2 =V_{f,1}+V_{f,1}^{(\sigma_2)}-
2V_{f,2} - (1+\sigma_2)V_{f,3} + (\sigma_4-2-\sigma_2^2)V_{f,4}.
\end{align*}

For higher moments we recall the following Wick type factorization
Lemma from \cite{MR3805203}.
\begin{lemma}\label{lemma:pairings}
For $k\geq2$ and $z_1,\dots,z_k\in\C$ with $z_l=x_l\pm i\eta_l$ and
$\eta_l>0$ we have that
\begin{align}\label{eq:pairings} \E[X(z_1)\dots X(z_k)]&= \sum_{\pi\in
P_2([k])}\prod_{\{a,b\}\in\pi}\E[X(z_a)X(z_b)]+\stO{\frac{1}{\sqrt
{N\boldsymbol{\eta}}}\sum_{a\not= b}\frac{1}{\sqrt{\eta_a}\Phi_{a,b}}},
\end{align}
where $[k]\defeq\{1,\dots,k\}$, $\boldsymbol\eta=\eta_1\dots\eta_k$,
$P_2(L)$ are the partitions of a set $L$ into subsets of size $2$ and
\[
\Phi_{a,b}\defeq\1_{\abs{x_a},\abs{x_b}\leq2}\left(\abs{\eta_a}+\abs
{\eta_b}+\abs{x_a-x_b}^2\right)+\left[(\abs{x_a}-2)_+ +(\abs{x_b}-2)_+
\right].
\]
\end{lemma}
The error term in \eqref{eq:pairings} is slightly stronger than that in
\cite{MR3805203} since the $\Phi_{a,b}$ includes a $\abs
{x_a-x_b}^2$. This strengthening follows along the lines of the
original proof by using the more precise analysis of the self
consistent equation outlined in Lemma \ref{lemma:EXX}. We check that
integrating the error term from \eqref{eq:pairings} over $(I_{\eta
_0}^M)^k$, with $\eta_0$ being chosen as above according to the
regularity of $f$, again gives asymptotically $N^{-1/2}$ in the case of
bounded $f'$ and $N^{-1/6}$ in the general case. By integrating the
Wick type product and using \eqref{deltafN2moment} we therefore arrive at
\begin{align}\label{deltafNkmoment}
\E\left(\widehat\Delta_{f}^{(N)}\right)^k = \E\left(\widehat\Delta
_f\right)^k +
\begin{cases}
\stO{N^{-1/2}} & \text{if }f'\in L^\infty([-3,3]),\\
\stO{N^{-1/6}} & \text{else.}
\end{cases}
\end{align}
We note that the error terms are implicitly $k$-dependent. By counting
the number of pair partitions we find that, to the leading order in
$N$, the implicit coefficients scale like $C^k (k/2)!$ with a constant
depending on $f$.

Recalling \eqref{DeltafN} and the definition of $T_f^{(N)}$ from \eqref{weak},
we conclude that the overall fluctuations have moments
\begin{align}\label{mom}
\E\left( T_f^{(N)} \right)^k
= \E\left(\widehat\Delta_f \right)^k
+
\begin{cases}
\landauO{C^k(k/2)!N^{-1/2}} & \text{if }f'\in L^\infty([-3,3]),\\
\landauO{C^k(k/2)!N^{-1/6}} & \text{else.}
\end{cases}
\end{align}

Let $\phi_N(t)$ denote the characteristic function of $T_f^{(N)}$ and
$\phi(t)$ the characteristic function of the
Gaussian variable $\widehat\Delta_f$. Then the moment bound \eqref
{mom} implies that
\[
\abs{ \phi_N(t)-\phi(t)}\le CN^{-1/6} t e^{Ct^2}
\]
with some constant $C$ depending on $f$. Using the well-known bound
(see, e.g., \cite[Theorem 1.4.13.]{MR1745554} and the
references therein)
\[
d_L(F, G) \le\frac{1}{\pi}\int_0^T \abs{ \phi_F(t)-\phi_G(t)} \frac
{\diff t}{t} + \frac{2e\log T}{T}
\]
for any two distributions $F$ and $G$ with characteristic functions
$\phi_F$ and $\phi_G$, we immediately obtain \eqref{speed}
by choosing $T=c\sqrt{\log N}$. This completes the proof of Theorem~\ref
{mainThmDiag}.

\section{Off-diagonal entries}
For the decomposition
\[
H=\left(
\begin{matrix}
h_{11} & h_{12} & h_1^* \\
h_{21} & h_{22} & h_2^* \\
h_1 & h_2 &\widehat H
\end{matrix}
\right)
\]
we find from the Schur complement formula that
\begin{align*}G(z)_{12}&=-\frac
{g_{12}}{g_{11}g_{22}-g_{12}g_{21}}=-m(z)^2 g_{12}+\stO{\frac{1}{N\abs
{\eta}}},
\end{align*}
where $g_{ij}\defeq h_{ij}-\delta_{ij}z-\braket{h_i,G(z)h_j}$.

We now set $Y(z)=Y^{(N)}(z)\defeq\sqrt{N}\braket{h_1,\widehat G(z)
h_2}$ and begin to compute
(all summation indices run from 3 to $N$)
\begin{align}\label{eqYY}
\E\left[ Y(z)Y(z')|\widehat H\right]&=N\sum_{a,b,c,d}\E\left[
h_{1a}\widehat{G}_{ab}h_{b2}h_{1c}\widehat{G}'_{cd}h_{d2}|\widehat
H\right] \\ &=\frac{\sigma_2^2}{N}\sum_{a,b} \widehat{G}_{ab}\widehat
{G}'_{ab} +\stO{\frac{\Psi}{N}}= \frac{\sigma_2^2m(z)m(z')}{1-\sigma
_2m(z)m(z')} + \stO{\frac{\Psi}{\sqrt N\Phi}}\nonumber
\end{align}
and
\begin{align}\label{eqYYt}
\E\left[ Y(z)\overline{Y(\overline{z'})}|\widehat H\right]&=N\sum
_{a,b,c,d}\E\left[ h_{1a}\widehat{G}_{ab}h_{b2}h_{2c}\widehat
{G}'_{cd}h_{d1}|\widehat H\right] \\ &=\frac{1}{N}\sum_{a,b} \widehat
{G}_{ab}\widehat{G}'_{ba} +\stO{\frac{\Psi}{N}}= \frac
{m(z)m(z')}{1-m(z)m(z')} + \stO{\frac{\Psi}{\sqrt N\Phi}}.\nonumber
\end{align}
For both estimates we made use of the fact the $h_{ab}$ are centered
and therefore have to appear at least twice to have non-zero
expectation. The main contribution comes from the pairing
$a=d$, $b=c$. Some exceptional pairings, such as the
four-pairing $a=b=c=d$, were incorporated in the error term by their
reduced combinatorics.
From Proposition \ref{prop_int_pleijel} we then find that
\begin{align*}f(H)_{12}&=\frac{1}{\pi}\iint_{I_{\eta_0}^M} m(z)^2\left
[\braket{h_1,\widehat G(z) h_2}-h_{12}\right] \diff\eta\diff f(x) +\stO
{\frac{\norm{\diff f}}{N}}.
\end{align*}
For the second term it follows, just as before, that
\[
\frac{1}{\pi}\iint_{I_{\eta_0}^M} m(z)^2 h_{12} \diff\eta\diff f(x)=
h_{12}\int f(x)x\, \mu_{sc}(\diff x)+\stO{\eta_0}.
\]
For the first term we set
\[
\widetilde\Delta_{f}^{(N)}\defeq\iint_{I_{\eta_0}^M} m(z)^2 Y(z)\diff
\eta\diff f(x)
\]
and by a computation analogous to \eqref{Eint3a} using \eqref{eqYY} and
an expansion of the form
\[
\frac{a^3}{1-a}=\frac{a}{1-a}-a-a^2
\]
we arrive at
\begin{align*}
\E\left(\widetilde\Delta_{f}^{(N)}\right)^2 = V_{f,1}^{(\sigma_2)} -
V_{f,2} -\sigma_2V_{f,3}+
\begin{cases}
\stO{N^{-1/2}} & \text{if }f'\in L^\infty([-3,3]),\\
\stO{N^{-1/6}} & \text{else.}
\end{cases}
\end{align*}
Similarly, from \eqref{eqYYt} we find that
\begin{align*}\E\abs{\widetilde\Delta_{f}^{(N)}}^2 &= V_{f,1} -
V_{f,2}-V_{f,3}+
\begin{cases}
\stO{N^{-1/2}} & \text{if }f'\in L^\infty([-3,3]),\\
\stO{N^{-1/6}} & \text{else.}
\end{cases}
.
\end{align*}
Finally, due to a Wick type theorem for $Y(z)$ which is proved along
the lines of Lemma \ref{lemma:pairings} we arrive at
\begin{align}
\E\left(S_f^{(N)}\right)^k \left(\overline{ S_f^{(N)}}\right)^l = \E
\left(\widetilde\Delta_{f}\right)^k \left(\overline{\widetilde\Delta
_{f}}\right)^l+
\begin{cases}
\stO{N^{-1/2}} & \text{if }f'\in L^\infty([-3,3]),\\
\stO{N^{-1/6}} & \text{else,}
\end{cases}
\end{align}
where $\widetilde\Delta_f$ is a centered complex Gaussian such that
\[
\E\widetilde\Delta_{f}^2 = V_{f,1}^{(\sigma_2)} - V_{f,2} -\sigma
_2V_{f,3},\quad\E\abs{\widetilde\Delta_{f}}^2 = V_{f,1} - V_{f,2}-V_{f,3}.
\]
We have proven Theorem \ref{mainThmOffDiag}.
\printbibliography%

\end{document}